\documentclass{birkmult}
\usepackage{amssymb}
\usepackage{amsmath}

\usepackage{diagrams}
\diagramstyle[centredisplay,dpi=600,UglyObsolete,heads=LaTeX]

\newcommand\Zset{\mathbb {Z}}
\newcommand\f{{\mathcal F}}
\newcommand\te{{\mathcal T}}
\newcommand\ef{{\mathfrak F}}
\newcommand\dirlim{\mathop{\varinjlim}\limits}
\newcommand\homo{\mathrm{Hom}}
\newcommand\ann{{\mathrm{ann}}}

\newtheorem{thm}{Theorem}[section]
\newtheorem{cor}[thm]{Corollary}

\newtheorem{prop}[thm]{Proposition}
\theoremstyle{definition}
\newtheorem{defn}[thm]{Definition}
\theoremstyle{remark}

\newtheorem*{ex}{Example}
\numberwithin{equation}{section}

\begin{document}

\title{A note on $(\alpha, \beta)$-higher derivations and their extensions to modules of quotients}

\author{Lia Va\v s}

\address{Department of Mathematics, Physics and Statistics\\ University of the Sciences
in Philadelphia\\ Philadelphia, PA 19104, USA}

\email{l.vas@usp.edu}

\author{Charalampos Papachristou}
\address{Department of Mathematics, Physics and Statistics\\ University of the Sciences
in Philadelphia\\ Philadelphia, PA 19104, USA}
\email{c.papach@usp.edu}



\date{April 29, 2009}

\begin{abstract}
We extend some recent results on the differentiability of torsion theories. In particular, we generalize the concept of $(\alpha, \beta)$-derivation to $(\alpha, \beta)$-higher derivation and demonstrate that a filter of a hereditary torsion theory that is invariant for $\alpha$ and $\beta$ is $(\alpha, \beta)$-higher derivation invariant. As a consequence, any higher derivation can be extended from a module to its module of quotients. Then, we show that any higher derivation extended to a module of quotients extends also to a module of quotients with respect to a larger torsion theory in such a way that these extensions agree. We also demonstrate these results hold for symmetric filters as well. We finish the paper with answers to two questions posed in  [L. Va\v s, Extending higher derivations to rings and modules of quotients, International Journal of Algebra, 2 (15) (2008), 711--731].
In particular, we present an example of a non-hereditary torsion theory that is not differential.
\end{abstract}

\maketitle

\section{Preliminaries and summary of results}
\label{section_preliminaries}

Recall that a {\em derivation} on a ring $R$ is an additive mapping $\delta: R \rightarrow R$ such that $\delta(rs)=\delta(r)s+r\delta(s)$ for all $r,s\in R.$ An additive mapping $d: M\rightarrow M$ on a right $R$-module $M$ is a {\em $\delta$-derivation} if $d(xr)=d(x)r+x\delta(r)$ for all $x\in M$ and $r\in R.$  If $\alpha$ and $\beta$ are ring automorphisms, the derivation concept generalizes to $(\alpha, \beta)$-derivation by requiring that $\delta(rs)=\delta(r)\alpha(s)+\beta(r)\delta(s)$ for all $r,s\in R.$

A {\em torsion theory} for $R$ is a pair $\tau = (\te, \f)$ of
classes of $R$-modules such that $\te$ and $\f$ are maximal classes having the property that
Hom$_R(T,F)=0,$ for all $T \in \te$ and $F \in \f.$
The modules in $\te$ are {\em torsion modules} and the modules in $\f$ are
{\em torsion-free modules}. For a torsion theory $\tau=(\te, \f)$, $\te(M)$ denotes the
largest torsion submodule of a right $R$-module $M$ and $\f(M)$ denotes the quotient $M/\te(M).$ $\tau = (\te, \f)$ is {\em hereditary} if $\te$ is closed under taking submodules (equivalently $\f$ is closed under formation of injective envelopes). If $\te(R)=0,$ $\tau$ is said to be faithful.

If $M$ is a right $R$-module with submodule $N$ and $m\in M,$ denote $\{r\in R\; | \; mr\in N\}$ by $(m:N).$ Then $(m:0)$ is the annihilator $\ann(m).$ A {\em Gabriel filter} $\ef$ on a ring $R$ is a nonempty collection of right ideals such that
\begin{enumerate}
\item If $I\in \ef$ and $r\in R,$ then $(r:I)\in \ef.$

\item If $I\in \ef$ and $J$ is a right ideal with $(r:J)\in
\ef$ for all $r\in I,$ then $J\in \ef$.
\end{enumerate}

If $\tau=(\te, \f)$ is a hereditary torsion theory, the collection of right
ideals $\{ I\; |\; R/I\in \te\;\}$ is a Gabriel filter. Conversely, if $\ef$ is a Gabriel
filter, then the class of modules $\{ M\; |\; \ann(m)\in \ef$ for every $m\in M\}$ is a torsion class of a
hereditary torsion theory.

If $\delta$ is any additive map on a ring $R$, a Gabriel filter $\ef$ is said to be {\em $\delta$-invariant} if for every $I\in \ef$ there is $J\in \ef$ such that $\delta(J)\subseteq I.$ If $\ef$ is $\delta$-invariant for all derivations $\delta,$ it is said to be a {\em differential filter}. The hereditary torsion theory determined by $\ef$ is said to be {\em differential} in this case. By Lemma 1.5 from \cite{Bland_paper}, $\ef$ is $\delta$-invariant iff $d(\te(M))\subseteq \te(M)$ for every right $R$-module $M$ and every $\delta$-derivation $d$ on $M.$ In \cite{Lia_Diff}, it is shown that Lambek, Goldie and any perfect hereditary torsion theories are differential. Lomp and van den Berg extend these results in \cite{Lomp_Berg} by showing that every Gabriel filter that is $\alpha$ and $\beta$-invariant is also $\delta$-invariant for any $(\alpha, \beta)$-derivation $\delta$ (Theorem 2, \cite{Lomp_Berg}). As a direct consequence, {\em every} hereditary torsion theory is differential (Corollary 3, \cite{Lomp_Berg}). This answers a question from \cite{Lia_Diff}.

If $\tau$ is a hereditary torsion theory with Gabriel filter $\ef$ and $M$ is a right $R$-module, the module of quotients $M_{\ef}$ of $M$ is defined as the largest submodule $N$ of the injective envelope $E(M/\te(M))$ of $M/\te(M)$ such that $N/(M/\te(M))$ is torsion module (i.e. the closure of $M/\te(M)$ in $E(M/\te(M))$). The $R$-module $R_{\ef}$ has a ring structure and $M_{\ef}$ has a structure of a right $R_{\ef}$-module (see exposition on pages 195--197 in \cite{Stenstrom}). The ring $R_{\ef}$ is called the right ring of quotients with respect to the torsion theory $\tau.$

Consider the map $q_M:M\rightarrow M_{\ef}$ obtained by composing the projection $M\rightarrow M/\te(M)$ with the injection $M/\te(M)\rightarrow M_{\ef}.$  This defines a left exact functor $q$ from the category of right $R$-modules to the category of right $R_{\ef}$-modules (see \cite{Stenstrom} pages 197--199).

In Theorem on page 277 and Corollary 1 on page 279 of \cite{Golan_paper}, Golan has shown that if $\ef$ is differential, then any $\delta$-derivation $d$ on any module $M$ extends to a derivation on the module of quotients $M_{\ef}$ such that $d q_M=q_M d.$
Bland proved that such extension is unique and that the converse is also true (Propositions 2.1 and 2.3 in \cite{Bland_paper}).
Thus a filter $\ef$ is differential iff every derivation on any module $M$ extends uniquely to a derivation on the module of quotients $M_{\ef}.$

The paper is organized as follows. In Section \ref{section_alpha_beta}, we generalize the concept of $(\alpha, \beta)$-derivation to $(\alpha, \beta)$-higher derivation (Definition \ref{definition_of_alpha-beta}).
In Section \ref{section_HD-invariant}, we show that every filter $\ef$ that is $\alpha$ and $\beta$-invariant is also $\Delta$-invariant for any $(\alpha, \beta)$-higher derivation $\Delta$ (Proposition \ref{alpha-beta-invariant}). As a consequence, we obtain that every Gabriel filter is higher differential (Corollary \ref{any_filter_is_HD}) and that every higher derivation on a module extends to its module of quotients (Corollary \ref{any_HD_extends}). In Section \ref{section_extending}, we show that the assumptions for some results from \cite{Lia_HD} and \cite{Lia_Extending} can be relaxed and that these results hold for {\em every} two filters $\ef_1$ and $\ef_2$ such that $\ef_1\subseteq\ef_2$ (Corollary \ref{agreement_HD}). In Section \ref{section_symmetric}, we show that the results from previous sections hold for symmetric filters as well (Corollary \ref{any_symmetric_filter_is_HD}). Lastly, in Section \ref{section_example}, we present an example of a torsion theory that is not differential (Example \ref{example_not_diff}) thus answering a question from \cite{Lia_HD}. Using result from Section \ref{section_HD-invariant}, we also show that there cannot exist a hereditary torsion theory that is differential but not higher differential.

\section{$(\alpha,\beta)$-higher derivations}
\label{section_alpha_beta}

Recall that a {\em higher derivation (HD)} on $R$ is an indexed family $\{\delta_n\}_{n\in\omega}$ of additive maps $\delta_n$ such that $\delta_0$ is the identity mapping on $R$ and
\[\delta_n(rs)=\sum_{i=0}^n \delta_i(r)\delta_{n-i}(s)\]
for all $n.$ For example, if $\delta$ is a derivation, the family $\{\frac{\delta^n}{n!}\}$ is a higher derivation.

Let $\alpha$ and $\beta$ be ring automorphisms. Throughout this, and most of the next section, we assume that $R$ is a ring in which $n1_R$ is invertible for every positive integer $n.$ In case that $\alpha$ and $\beta$ are both identities, we can drop this additional assumption on $R.$ We generalize the concept of an $(\alpha,\beta)$-derivation to higher derivations as follows.

\begin{defn}  An {\em $(\alpha,\beta)$-higher derivation ($(\alpha,\beta)$-HD)} on $R$ is an indexed family $\Delta=\{\delta_n\}_{n\in\omega}$ of additive maps $\delta_n$ such that $\delta_0$ is the identity mapping on $R$ and
\[\delta_n(rs)=\delta_n(r)\alpha^n(s)+\sum_{i=1}^n\frac{i!(n-i)!}{n!}\sum_{k_0+\ldots+k_i=n-i} \delta_{k_0} \prod_{j=1}^i \beta\delta_{k_j}(r)\;\alpha^{k_0} \prod_{j=1}^i \delta_1\alpha^{k_j}(s)\]
where any composition of the form $\delta_1^j\delta_1^k$ in the second product is substituted by $\delta_{j+k}.$ Also, in case that $\beta$ is the identity, $\delta_j\delta_k$ in the first product is substituted by $\delta_{j+k}.$
\label{definition_of_alpha-beta}
\end{defn}

For $n=1$ this formula yields the familiar $\delta_1(rs)=\delta_1(r)\alpha(s)+\beta(r)\delta_1(s).$ For $n=2$ we obtain that $\delta_2(rs)=\delta_2(r)\alpha^2(s)+\frac{1}{2}\beta\delta_1(r)\delta_1\alpha(s)+\frac{1}{2}\delta_1\beta(r)\alpha\delta_1(s)+\beta^2(r)\delta_2(s).$

Note that the elements of the form $(n1_R)^{-1}$ are in the center of $R$ for any positive integer $n$ since elements of the form $n1_R$ are in the center of $R.$ So, the coefficients $\frac{i!(n-i)!}{n!}$ commute with all ring elements.

If $\alpha$ is an identity, we obtain
\[\delta_n(rs)=\delta_n(r)s+\sum_{i=1}^n\frac{i!(n-i)!}{n!}\sum_{k_0+\ldots+k_i=n-i} \delta_{k_0} \prod_{j=1}^i \beta\delta_{k_j}(r)\; \delta_i(s)\]

If $\beta$ is an identity, we obtain
\[\delta_n(rs)=\delta_n(r)\alpha^n(s)+\sum_{i=1}^n\frac{i!(n-i)!}{n!}\sum_{k_0+\ldots+k_i=n-i} \delta_{n-i}(r)\;\alpha^{k_0} \prod_{j=1}^i \delta_1\alpha^{k_j}(s)\]

In particular, if both $\alpha$ and $\beta$ are identities, we obtain
\[\delta_n(rs)=\delta_n(r)s+\sum_{i=1}^n\frac{i!(n-i)!}{n!}\sum_{k_0+\ldots+k_i=n-i} \delta_{n-i}(r)\;\delta_{i}(s)\]
\[=\delta_n(r)s+\sum_{i=1}^n\frac{i!(n-i)!}{n!}\;\delta_{n-i}(r)\;\delta_{i}(s)\sum_{k_0+\ldots+k_i=n-i} 1_R=\]
\[=\delta_n(r)s+\sum_{i=1}^n\frac{i!(n-i)!}{n!}\delta_{n-i}(r)\;\delta_{i}(s)\frac{n!}{i!(n-i)!}=\]
\[\delta_n(r)s+\sum_{i=1}^n\frac{i!(n-i)!}{n!}\frac{n!}{i!(n-i)!}\delta_{n-i}(r)\;\delta_{i}(s)\;=\;\sum_{i=0}^n \delta_{n-i}(r)\delta_{i}(s).\]
The last formula in the chain above is exactly the one that defines a higher derivation.

\section{Higher differentiation invariance}
\label{section_HD-invariant}

If $\Delta=\{\delta_n\}$ is an $(\alpha,\beta)$-HD, a Gabriel filter $\ef$ is  {\em$\Delta$-invariant} if for every $I\in \ef$ and every $n,$ there is $J\in \ef$ such that $\delta_i(J)\subseteq I$ for all $i\leq n$ (equivalently, for every $I\in \ef$ and every $n,$ there is $J\in \ef$ such that $\delta_n(J)\subseteq I$). If a filter $\ef$ is $\Delta$-invariant for every $(\alpha,\beta)$-HD $\Delta,$ $\ef$ is said to be {\em higher differential (HD)}. The hereditary torsion theory determined by $\ef$ is said to be {\em higher differential} in this case.

\begin{prop}
Let $\Delta$ be a higher $(\alpha, \beta)$-derivation. Then any Gabriel filter $\ef$ that is $\alpha$ and $\beta$-invariant is $\Delta$-invariant.
\label{alpha-beta-invariant}
\end{prop}
\begin{proof}
Let $I\in \ef.$ We shall use induction to show that for every $n,$ there is $J\in \ef$ such that $\delta_n(J)\subseteq I.$

For $n=0$ the claim trivially holds for $J=I.$ Assume that the claim holds for all $i<n.$ By induction hypothesis for $I$ there are right ideals $J_i\in \ef$ with $\sum_{k_0+\ldots+k_i=n-i} \delta_{k_0} \prod_{j=1}^i \beta\delta_{k_j}(J_i)\subseteq I$ for all $0<i\leq n.$ Note that for $i=n,$ $\sum_{k_0+\ldots+k_i=n-i} \delta_{k_0} \prod_{j=1}^i \beta\delta_{k_j}$ is $\beta^n.$ Since $\ef$ is $\beta$-invariant, there is $J_n\in \ef$ such that $\beta^n(J_n)\subseteq I.$

Take $J_0=I,$ and let $J_{\alpha}$ be a right ideal in $\ef$ with $\alpha^n(J_{\alpha})\subseteq I.$ Let $K=\bigcap_{i\leq n} J_i\cap J_{\alpha}.$ Then $K$ is in $\ef$ and $K\subseteq I.$ Define $J=\{r\in K| \delta_n(r)\in I\}.$ Then $J$ is a right ideal of $R,$ $J\subseteq K\subseteq I$ and $\delta_n(J)\subseteq I.$ Also, \[\sum_{k_0+\ldots+k_i=n-i} \delta_{k_0} \prod_{j=1}^i \beta\delta_{k_j}(J)\subseteq \sum_{k_0+\ldots+k_i=n-i} \delta_{k_0} \prod_{j=1}^i \beta\delta_{k_j}(J_i)\subseteq I\] for all $0<i\leq n.$ In order to prove that $J$ is in $\ef,$ it is sufficient to show that $(r:J)\in \ef$ for all $r\in K.$ To show that, we shall show that  $(\alpha^{-n}\delta_n(r):K)\subseteq (r:J).$ Since $K\in \ef,$  $(\alpha^{-n}\delta_n(r):K)\in \ef,$ and so this will be sufficient for $(r:J)\in\ef.$

Let $s\in(\alpha^{-n}\delta_n(r):K).$ Then $\delta_n(r)\alpha^n(s)\in \alpha^n(K)\subseteq \alpha^n(J_{\alpha})\subseteq I.$  The terms
$\sum_{k_0+\ldots+k_i=n-i}\delta_{k_0} \prod_{j=1}^i \beta\delta_{k_j}(r)$ are in $I$ for every $i=1,\ldots, n$ by construction.
Since fractions $\frac{i!(n-i)!}{n!}$ are in the center of $R,$ we obtain that every term on the right side of the formula below is in $I$ as well.
\[\delta_n(rs)=\delta_n(r)\alpha^n(s)+ \sum_{i=1}^n\frac{i!(n-i)!}{n!}\sum_{k_0+\ldots+k_i=n-i} \delta_{k_0} \prod_{j=1}^i \beta\delta_{k_j}(r)\;\alpha^{k_0} \prod_{j=1}^i \delta_1\alpha^{k_j}(s).\]
Thus $\delta_n(rs)\in I.$ Since $rs\in K,$ we have that $rs$ is in $J.$ So $s\in(r:J).$
\end{proof}

For the remainder of the paper, we drop the condition that the integer multiples of $1_R$ are invertible and we work with a most general unital ring. Recall that if $\alpha$ and $\beta$ are identities, the formula in Definition \ref{definition_of_alpha-beta} becomes $\delta_n(rs)=\sum_{i=0}^n \delta_{n-i}(r)\delta_{i}(s).$ So, the assumption on the invertibility of the integer multiples of $1_R$ is no longer needed. Note that the proof of Proposition \ref{alpha-beta-invariant} still holds in this case as well. Thus, as a direct corollary of Proposition \ref{alpha-beta-invariant}, we obtain the following.

\begin{cor} Any Gabriel filter is higher derivation invariant (i.e. every torsion theory is higher differential).
\label{any_filter_is_HD}
\end{cor}

Let $\Delta=\{\delta_n\}_{n\in\omega}$ be a higher derivation on $R.$ If $\{d_n\}_{n\in\omega}$ is an indexed family of additive maps on a right $R$-module $M$ such that $d_0$ is the identity mapping on $M$ and $d_n(mr)=\sum_{i=0}^n d_i(m)\delta_{n-i}(r)$ for all $n,$ we say that $\{d_n\}$ is {\em higher $\Delta$-derivation} ($\Delta$-HD for short) on $M.$ If $D$ is such that every $d_n$ extends to the module of quotients $M_{\ef}$ of a Gabriel filter $\ef$ such that $d_n q_M=q_M d_n$ for all $n,$ then we say that $D$ {\em extends} to a $\Delta$-HD on $M_{\ef}.$

\begin{cor} Let $\tau$ be a hereditary torsion theory with filter $\ef$ and $\Delta$ be a HD on $R.$ Every $\Delta$-HD $D$ on any module $M$ extends uniquely to the module of quotients $M_{\ef}.$
\label{any_HD_extends}
\end{cor}
\begin{proof}
Bland showed that a Gabriel filter is a HD filter iff for every $R$-module $M,$ every HD $\{d_n\}$ on $M,$ $d_i(\te(M))\subseteq\te(M)$ for all $i\leq n$ for all $n$ (Lemma 3.5 in \cite{Bland_paperHD}). Bland also showed that $\tau$ is higher differential iff every $\Delta$-HD $D$ on any module $M$ extends uniquely to a $\Delta$-HD on $M_{\ef}$ (Proposition 4.2, \cite{Bland_paperHD}). Since every hereditary torsion theory is higher differential by Corollary \ref{any_filter_is_HD}, the result follows.
\end{proof}

\section{Extending derivation to different modules of quotients}
\label{section_extending}

Let $\ef_1$ and $\ef_2$ be two filters such that $\ef_1\subseteq \ef_2.$ Let $M$ be a right $R$-module, $q_i$ the natural maps $M\rightarrow M_{\ef_i}$ for $i=1,2,$ and $q_{12}$ the map $M_{\ef_1}\rightarrow M_{\ef_2}$ induced by inclusion  $\ef_1\subseteq \ef_2.$ In this case, $q_{12}q_1=q_2.$ Let $d$ be a $\delta$-derivation on $M.$ If the diagram below commutes, we say that the extensions of $d$ on $M_{\ef_1}$ and $M_{\ef_2}$ {\em agree.}

\begin{diagram}
  &             & M_{\ef_1}    &      & \rTo^{d} &             & M_{\ef_1}\\
  & \ruTo^{q_1} & \vLine       &      &            & \ruTo^{q_1} &    \\
M &             & \rTo^{d}     &      & M          &             & \dTo_{q_{12}}\\
  & \rdTo^{q_2} & \dTo_{q_{12}}&      &            & \rdTo^{q_2} &     \\
  &             & M_{\ef_2}    &      & \rTo^{d} &             & M_{\ef_2}\\
\end{diagram}

Let $d$ be a derivation on $M$ that extends to $M_{\ef_1}.$ Conditions under which $d$ can be extended to $M_{\ef_2}$ so that the extension agree were studied in \cite{Lia_Extending}.

By Corollary 3 from \cite{Lomp_Berg}, any derivation $d$ on $M$ extends to both $M_{\ef_1}$ and $M_{\ef_2}.$ By Proposition 2 from \cite{Lia_Extending}, the two extensions agree then. This implies that, for any module $M,$ the extensions of a $\delta$-derivation $d$ to $M_{\ef_1}$ and $M_{\ef_2}$ agree. In particular, the extension of $d$ on module of quotients with respect Lambek or Goldie torsion theory agree with the extension of $d$ with respect to any other hereditary and faithful torsion theory.

In \cite{Lia_HD}, the concept of agreeing extensions is generalized to higher derivations. Let  $\ef_1$ and $\ef_2$ be two filters such that $\ef_1\subseteq \ef_2,$ $\Delta$ a HD on $R,$ $M$ be a right $R$-module, and $\{d_n\}$ a $\Delta$-HD defined on $M.$ If $\{d_n\}$ extends to $M_{\ef_1}$ and $M_{\ef_2}$ in such a way that the following diagram commutes for every $n,$ then we say that the extensions of $\{d_n\}$ on $M_{\ef_1}$ and $M_{\ef_2}$ {\em agree.}
\begin{diagram}
  &             & M_{\ef_1}    &      & \rTo^{d_n} &             & M_{\ef_1}\\
  & \ruTo^{q_1} & \vLine       &      &            & \ruTo^{q_1} &    \\
M &             & \rTo^{d_n}     &      & M          &             & \dTo_{q_{12}}\\
  & \rdTo^{q_2} & \dTo_{q_{12}}&      &            & \rdTo^{q_2} &     \\
  &             & M_{\ef_2}    &      & \rTo^{d_n} &             & M_{\ef_2}\\
\end{diagram}

\begin{cor}
If $\ef_1$ and $\ef_2$ are two filters such that $\ef_1\subseteq \ef_2,$ and $\Delta$ a HD on $R,$ then for any module $M$ the extension of a $\Delta$-HD $D$ to $M_{\ef_1}$ agrees with the extension of $D$ to $M_{\ef_2}.$ In particular, the extension of $D$ on module of quotients with respect Lambek or Goldie torsion theory agree with the extension of $D$ with respect to any other hereditary and faithful torsion theory.
\label{agreement_HD}
\end{cor}
\begin{proof}
By Proposition 3 from \cite{Lia_HD}, if a $\Delta$-HD $D$ extends to $M_{\ef_1}$ and $M_{\ef_2},$ then the extensions agree. By Corollary \ref{any_HD_extends}, $D$ always extends to both $M_{\ef_1}$ and $M_{\ef_2}$ and so the result follows.
\end{proof}

\section{Symmetric modules of quotients}
\label{section_symmetric}

In \cite{Lia_Extending}, the concept of invariant filters is extended to symmetric filters as well. A symmetric filter $_l\ef_r$ induced by a left filter $\ef_l$ and a right filter $\ef_r$ can be defined so that the hereditary torsion theory $_l\tau_r$ on $R$-bimodules that correspond to $_l\ef_r$ has the torsion class equal to the intersection of torsion classes of $\tau_l$ and $\tau_r:$
\[_l\te_r=\,\te_l\cap\te_r.\]

In \cite{Ortega_paper}, the {\em symmetric module of quotients} $_{\ef_l}M_{\ef_r}$ of $M$ with respect to $_l\ef_r$ is defined to be
\[_{\ef_l}M_{\ef_r}=\dirlim_{K\in _l\ef_r}\;  \homo(K, \frac{M}{_l\te_r(M)})\]
where the homomorphisms in the formula are $R\otimes R^{op}$ homomorphisms (equivalently $R$-bimodule homomorphisms).
We shorten the notation $_{\ef_l}M_{\ef_r}$ to $_lM_r.$ Just as in the right-sided case, there is a left exact functor $q_M$ mapping $M$ to the symmetric module of quotients $_lM_r$ such that $\ker q_M$ is the torsion module $_l\te_r(M).$

Every derivation $\delta$ on $R$ determines a derivation on $R\otimes_{\Zset}R^{op}$ given by $\overline{\delta}(r\otimes s)=\delta(r)\otimes s+r\otimes \delta(s).$ Similarly, every HD $\Delta$ on $R$ determines a HD $\overline{\Delta}$ on $R\otimes_{\Zset}R^{op}$ given by \[\overline{\delta_n}(r\otimes s)=\sum_{i=0}^n \delta_i(r)\otimes \delta_{n-i}(s).\]

If $M$ is an $R$-bimodule, and $\delta$ a derivation on $R$, we say that an additive map $d: M\rightarrow M$ is a {\em $\delta$-derivation} if \[d(xr)=d(x)r+x\delta(r)\mbox{ and }d(rx)=\delta(r)x+rd(x)\]
for all $x\in M$ and $r\in R.$ Note that $d$ is a $\overline{\delta}$-derivation on $M$ considered as a right $R\otimes_{\Zset}R^{op}$-module. Conversely, every $\overline{\delta}$-derivation of a right $R\otimes_{\Zset}R^{op}$-module determines a $\delta$-derivation of the corresponding bimodule.
Thus, every derivation $\delta$ on $R$ is a $\overline{\delta}$-derivation on $R$ considered as a right $R\otimes_{\Zset}R^{op}$-module. Conversely, every derivation $\overline{\delta}$ on $R\otimes_{\Zset}R^{op}$ is a $\delta$-derivation of $R\otimes_{\Zset}R^{op}$ considered as an $R$-bimodule.

This generalizes to higher derivations as well.
If $M$ is an $R$-bimodule, and $\Delta$ a HD on $R$, we say that an indexed family of additive maps $\{d_n\}$ defined on $M$ is a {\em $\Delta$-HD} if $d_0$ is an identity,
\[d_n(xr)=\sum_{i=0}^n\delta_i(x)\delta_{n-i}(r)\mbox{ and }d_n(rx)=\sum_{i=0}^n\delta_i(r)\delta_{n-i}(x)\]
for all $x\in M$ and $r\in R.$ It is straightforward to check that $\{d_n\}$ is a $\overline{\Delta}$-HD on $M$ considered as a right $R\otimes_{\Zset}R^{op}$-module. Conversely, every $\overline{\Delta}$-HD on a right $R\otimes_{\Zset}R^{op}$-module $M$ determines a $\Delta$-HD on $M$ considered as an $R$-bimodule. Specifically, every HD $\Delta$ on $R$ is a $\overline{\Delta}$-HD on $R$ considered as a right $R\otimes_{\Zset}R^{op}$-module. Conversely, every HD $\overline{\Delta}$ on $R\otimes_{\Zset}R^{op}$ is a $\Delta$-HD on $R\otimes_{\Zset}R^{op}$ considered as an $R$-bimodule.

A symmetric filter $_l\ef_r$ induced by a left Gabriel filter $\ef_l$ and a right Gabriel filter $\ef_r$ is said to be $\delta$-invariant if for every $I\in\, _l\ef_r,$ there is $J\in\, _l\ef_r$ such that $\overline{\delta}(J)\subseteq I.$
If we consider the right $R\otimes_{\Zset}R^{op}$-ideals $I$ and $J$ as $R$-bimodules, the condition $\overline{\delta}(J)\subseteq I$ is equivalent with $\delta(J)\subseteq I$ by the above observations. This definition generalizes to higher derivations on a straightforward way just as in one-sided case.

We say that $_l\ef_r$ is a {\em differential filter} if it is $\delta$-invariant for all derivations $\delta.$  The hereditary torsion theory determined by $_l\ef_r$ is said to be {\em differential} in this case. Similarly, a HD symmetric filter is defined.

The following proposition proves Corollaries \ref{any_filter_is_HD}, \ref{any_HD_extends} and \ref{agreement_HD} for symmetric filters.

\begin{cor}
\begin{enumerate}
\item Any symmetric Gabriel filter is higher derivation invariant (i.e. every symmetric torsion theory is higher differential).

\item Let $_l\tau_r$ be a symmetric hereditary torsion theory with filter $_l\ef_r$ and $\Delta$ be a HD on $R.$ Every $\Delta$-HD $D$ on any module $M$ extends uniquely to the module of quotients $_lM_r.$

\item  If $_l\ef_r^1$ and $_l\ef_r^2$ are two symmetric filters such that $_l\ef_r^1\subseteq\, _l\ef_r^2,$ and $\Delta$ is a HD on $R,$ then for any bimodule $M,$ the extensions of any $\Delta$-HD $D$ to $_lM_r^1$ and $_lM_r^2$ agree.
\end{enumerate}
\label{any_symmetric_filter_is_HD}
\end{cor}
\begin{proof}
1.  By Proposition 3 of \cite{Lia_Extending}, if $\ef_l$ and $\ef_r$ are differential, then $_l\ef_r$ is also differential. By Corollary 3 of \cite{Lomp_Berg}, $\ef_l$ and $\ef_r$ are always differential and so we obtain that every symmetric filter is differential as well. In \cite{Lia_HD}, the results on symmetric filters from \cite{Lia_Extending} are generalized to higher derivations. In particular, it is shown that if $\ef_l$ and $\ef_r$ are HD, that the symmetric filter $_l\ef_r$ is HD as well (see Proposition 4 of \cite{Lia_HD}). This, together with Corollary \ref{any_filter_is_HD}, gives us part 1.

2. Part iii) of Proposition 4 in \cite{Lia_HD} states that part 2. holds provided that the filter $_l\ef_r$ is HD. However, any filter is HD by part 1., so the result follows.

3. Proposition 5 in \cite{Lia_HD} states that part 3. holds provided that $_l\ef_r^1$ and $_l\ef_r^2$ are HD. Since these conditions are always fulfilled by part 1., the result follows.
\end{proof}

\section{Torsion theory that is not differential}
\label{section_example}

Gabriel filter, right ring and modules of quotients and related concepts are defined just for a torsion theory that is hereditary. Thus, if we want to generalize the concept of differential torsion theory to torsion theories that are not necessarily hereditary, we cannot use the second and third of the following three equivalent conditions for differentiability of a hereditary torsion theory.
\begin{enumerate}
\item For every derivation $\delta,$ and every right $R$-module $M$ with a $\delta$-derivation $d,$ $d(\te M)\subseteq \te M;$

\item Gabriel filter $\ef$ is $\delta$-invariant (i.e. for every $I\in \ef$ there is $J\in \ef$ with $\delta(J)\subseteq I$) for every derivation $\delta$;

\item For every derivation $\delta,$ every $\delta$-derivation on any module uniquely extends to the module of quotients.
\end{enumerate}
Note that just the first condition is meaningful even if a torsion theory is not hereditary. So, let us introduce the following definition.

\begin{defn}
Let $\tau$ be a (not necessarily hereditary) torsion theory. Then $\tau$ is {\em differential} if $d(\te M)\subseteq \te M$ for any ring derivation $\delta$ and any right $R$-module $M$ with a $\delta$-derivation $d.$
\end{defn}

The following example shows that not every torsion theory is differential.

\begin{ex}
Let $R=\Zset[x]$ and $I=(x).$ Consider the module $R/I\cong \Zset.$ Consider the class of right $R$-modules  $\te=\{M|\ker(M\rightarrow M\otimes_R \Zset)=M\}.$ Note that this class is closed under quotients, extensions and direct sums, so it defines a torsion class of a torsion theory $\tau.$ In this torsion theory $\te M=\ker(M\rightarrow M\otimes_R \Zset)$ for every module $M.$

Note that $\ker(R\rightarrow R\otimes_R \Zset)= I$ so $\te R=I.$ $\ker(I\rightarrow I\otimes_R\Zset)=I^2.$ This shows that
$I^2=\te I\neq I\cap \te R=I,$ so the torsion theory is not hereditary. Note also that if $\Zset$ were flat as a left $R$-module, then this torsion theory would necessarily have been hereditary.

Now let us consider the map $\delta: R\rightarrow R$ given by $\delta=\frac{d}{dx}$ i.e.
\[\delta(a_nx^n+a_{n-1}x^{n-1}+\ldots +a_1x+a_0)=n a_nx^{n-1}+(n-1)a_{n-1}x^{n-2}+\ldots +a_1.\]  It is easy to see that this is a derivation on $R.$ As $\delta(x)=1,$ we have that $x\in \te R$ and $\delta(x)\notin \te R.$
\qed
\label{example_not_diff}
\end{ex}

This answers the first of the three questions from section 6 of \cite{Lia_HD}: there is a non-hereditary and non-differential torsion theory.

The second question from \cite{Lia_HD} is asking if an extension of a derivation to module of quotients with respect to larger torsion theory can be restricted to extension of a derivation with respect to a smaller torsion theory. The affirmative answer emerged with Corollary 3 of \cite{Lomp_Berg} and Proposition 2 of \cite{Lia_Extending}. Namely, by Corollary 3 of \cite{Lomp_Berg}, every torsion theory is differential. By Proposition 2 of \cite{Lia_Extending}, this implies that all extensions of derivations to module of quotients agree. Moreover, by results of this paper, this result holds for higher derivations as well.

Finally, the third question from \cite{Lia_HD} is asking if there is a differential hereditary torsion theory that is not higher differential. By Corollary \ref{any_filter_is_HD} any hereditary torsion theory is higher differential so the answer to this question is ``no''.

\end{document}